\documentclass[11pt,a4paper]{amsart}
\usepackage[numbers,sort&compress]{natbib}
\usepackage[lmargin=40mm,rmargin=40mm,bmargin=30mm,tmargin=30mm]{geometry}
\newtheorem{theorem}{Theorem}

\renewcommand{\baselinestretch}{1.25}
\begin{document}

\title{Folding = Colouring}
\author{David R. Wood}
\thanks{Departament de Matem{\`a}tica Aplicada II, Universitat Polit{\`e}cnica de Catalunya, Barcelona, Spain (\texttt{david.wood@upc.es}). Research supported by a Marie Curie Fellowship of the European Community under contract MEIF-CT-2006-023865, and by the projects MEC MTM2006-01267 and DURSI 2005SGR00692.}

\date{\today}
\subjclass{05C15 (coloring of graphs and hypergraphs)}

\keywords{graph, folding, folding number, colouring, coloring, chromatic number, Hadwiger's conjecture}
\maketitle

The foldings of a graph\footnote{All graphs considered are finite, undirected, loopless, connected, and possibly have parallel edges. A graph is \emph{complete} if there is at least one edge between every pair of vertices.} are defined as follows. First, a graph $G$ is a \emph{folding} of itself. Now suppose that $G$ is not complete. Let $G'$ be a graph obtained from $G$ by identifying two vertices at distance $2$ in $G$. Then every folding of $G'$ is a \emph{folding} of $G$. The \emph{folding number} of $G$, denoted by $f(G)$, is the minimum order of a complete folding of $G$. These concepts were introduced by \citet{GGR-GC02}, who determined the folding number of wheels and fans. Here we prove the following more general theorem, where $\chi(G)$ is the chromatic number of $G$.

\begin{theorem}
For every graph $G$, $$f(G)=\chi(G).$$
\end{theorem}

\begin{proof}
First we prove that $\chi(G)\leq f(G)$. We proceed by induction on $|V(G)|$. In the base case, if $G$ is complete, then $f(G)=|V(G)|=\chi(G)$, and we are done. Otherwise, by the definition of folding-number, there is a graph $G'$ obtained from $G$ by identifying two vertices $v$ and $w$ at distance $2$ in $G$, such that $f(G')=f(G)$. By induction, $\chi(G')\leq f(G')$. If $v$ and $w$ are identified into a vertex $x$, then the colour assigned to $x$ in a colouring of $G'$ can be assigned to both $v$ and $w$ (since $v$ and $w$ are not adjacent, and every neighbour of $v$ or $w$ in $G$ is a neighbour of $x$ in $G'$). Thus $\chi(G)\leq\chi(G')\leq f(G')=f(G)$. 

We now prove that $f(G)\leq\chi(G)$, again by induction on $|V(G)|$. In the base case, $G$ is complete and $f(G)=|V(G)|=\chi(G)$. Now assume $G$ is not complete. 

We claim that $G$ has a $\chi(G)$-colouring such that some pair of vertices $v$ and $w$ at distance $2$ receive the same colour. First suppose that $G$ is an odd cycle. Then $\chi(G)=3$. Colour one vertex `red', and colour the remaining vertices `blue' and `green' alternately around the cycle. This defines the claimed colouring. Now assume that $G$ is not an odd cycle. Take any $\Delta(G)$-colouring of $G$. We have $\chi(G)\leq\Delta(G)$ by Brooks' Theorem \citep{Brooks41}, where $\Delta(G)$ is the maximum degree of $G$. Let $u$ be a vertex of degree $\Delta(G)$. Two neighbours $v$ and $w$ of $u$ receive the same colour, as otherwise we have $\Delta(G)+1$ colours. Now $v$ and $w$ are not adjacent since they receive the same colour. Thus $v$ and $w$ are at distance $2$. This proves the claim.

Let $G'$ be the graph obtained from $G$ by identifying $v$ and $w$. Since the $\chi(G)$-colouring of $G$ defines a colouring of $G'$, it follows that $\chi(G')\leq\chi(G)$. By induction, $f(G')\leq\chi(G')$. Since $G'$ is a folding of $G$, it follows that $f(G)\leq f(G')$. Therefore $f(G)\leq f(G')\leq \chi(G')\leq \chi(G)$.
\end{proof}

Folding is similar to edge-contraction (that is, identify two adjacent vertices, and delete the resulting loop). Perhaps this is relevant for Hadwiger's Conjecture which, in light of Theorem~1, states that the minimum order of a complete folding of a graph $G$ is at most the maximum order of a complete contraction of $G$.

Theorem~1 suggests the following algorithm for colouring a graph $G$: \\
\hspace*{6mm}\emph{While $G$ is not complete, identify two vertices at distance $2$, and repeat.}\\
First observe that this algorithm $2$-colours every bipartite graph. Theorem~1 implies that there exists a choice of vertices to identify at each step so that this algorithm optimally colours the given graph\footnote{A similar property holds for the greedy sequential colouring algorithm: Every graph admits a linear ordering of its vertices, so that if the vertices are greedily coloured in this order, then an optimal colouring is obtained \citep{Biggs}.}. Of course, choosing the right vertices is difficult\footnote{Indeed, \citet{GGR-GC02} showed that the complete graph on approximately $\sqrt{n/2}$ vertices is a folding of the wheel graph on $n$ vertices.}. A natural heuristic is to choose a pair of vertices with the maximum number of common neighbours. Then the maximum possible number of parallel edges are created when the chosen vertices are identified. But this heuristic says nothing for graphs with girth at least $5$.

\subsection*{Note: }

Since this paper was first released, I have discovered that graph foldings were actually introduced by \citet{CookEvans} in 1979, who proved Theorem~1 with an identical proof to that presented here. Foldings have since been studied by a number of authors \citep{BS81,Evans-JGT86,GaoHahn-DM95,BS-CN81}.


\end{document}